\documentclass[10pt,reqno]{amsart}
\usepackage[utf8]{inputenc}
\usepackage{amsopn,amssymb,mathrsfs,xcolor,url,enumerate,enumitem,accents}
\usepackage[abbrev]{amsrefs} 
\usepackage[a4paper,lmargin=3cm,rmargin=3cm,tmargin=4cm,bmargin=4cm,marginparwidth=2.8cm,marginparsep=1mm]{geometry} 
\usepackage[bookmarks=true,hyperindex,pdftex,colorlinks,citecolor=blue]{hyperref} 
\usepackage{cleveref}

\usepackage{marginnote} 

\newcommand{\restrict}{\mathord{\upharpoonright}}           

\newtheorem{theorem}{Theorem}[section]

\newtheorem{lemma}[theorem]{Lemma}
\newtheorem{proposition}[theorem]{Proposition}

\theoremstyle{definition}
\newtheorem{definition}[theorem]{Definition}
\newtheorem{remark}[theorem]{Remark}

\newcommand{\nm}[1]{\vert\kern-0.25ex\vert #1 \vert\kern-0.25ex\vert}

\newcommand{\R}{\mathbb{R}}
\newcommand{\N}{\mathbb{N}}
\newcommand{\diam}{\textnormal{diam}}
\newcommand{\free}[1]{\mathcal{F}(#1)}
\newcommand{\Lipo}[1]{\textnormal{Lip}_0(#1)}

\newcommand{\tildem}{\accentset{\sim}}
\newcommand{\cl}{\textnormal{cl}}
\newcommand{\supp}{\textnormal{supp}}
\newcommand{\trker}{\textnormal{trker\hspace{1.5 pt}}}

\newcommand{\cM}{\mathcal{M}}
\newcommand{\cB}{\mathcal{B}}
\newcommand{\cP}{\mathcal{P}}
\newcommand{\cA}{\mathcal{A}}
\newcommand{\cS}{\mathcal{S}}

\newcommand{\cU}{\mathcal{U}}
\newcommand{\cV}{\mathcal{V}}
\newcommand{\cW}{\mathcal{W}}
\newcommand{\cR}{\mathcal{R}}

\DeclareMathOperator{\Lip}{Lip}
\DeclareMathOperator{\inter}{int}

\numberwithin{equation}{section}

\allowdisplaybreaks

\title[Free spaces over strongly countable-dimensional spaces and APs]{Lipschitz-free spaces over strongly countable-dimensional spaces and approximation properties}
\author[F. Talimdjioski]{Filip Talimdjioski}
\address[F. Talimdjioski]{School of Mathematics and Statistics, University College Dublin, Belfield, Dublin 4, Ireland}
\email{filip.talimdjioski@ucdconnect.ie}

\date{\today}

\begin{document}

\begin{abstract}
	Let $T$ be a compact, metrisable and strongly countable-dimensional topological space. Let $\cM^T$ be the set of all metrics $d$ on $T$ compatible with its topology, and equip $\cM^T$ with the topology of uniform convergence, where the metrics are regarded as functions on $T^2$. We prove that the set $\cA^{T,1}$ of metrics $d\in\cM^T$ for which the Lipschitz-free space $\free{T,d}$ has the metric approximation property is residual in $\cM^T$.
\end{abstract}

\maketitle

\section{Introduction}\label{sect_intro}

For a metric space $(M,d)$ and a base point $x_0\in M$ we define the Banach space $\Lipo{M,x_0}$ consisting of all real-valued Lipschitz functions $f$ on $M$ that vanish at $x_0$, equipped with the norm $$\nm{f} := \sup_{x,y\in M, x\not= y} \frac{|f(x)-f(y)|}{d(x,y)}.$$ For any $x\in M$ we define the bounded linear functional $\delta_x\in \Lipo{M,x_0}^*$ by $\delta_x(f) = f(x)$, $f\in\Lipo{M,x_0}$. The closed linear span of the set $\{\delta_x : x\in M\}$ is called the \emph{Lipschitz-free space} $\free{M,x_0}$. It is well-known that $\Lipo{M,x_0}$ is isometric to the dual space of $\free{M,x_0}$ and that the isometric structure of both spaces does not depend on the choice of the base point $x_0$. We will hereafter write $\Lipo{M}$ and $\free{M}$ without specifying the base point, and call $\free{M}$ simply the free space over $M$. In the book \cite{weaver}, Weaver provides a comprehensive introduction to Lipschitz and Lipschitz-free spaces. In it, the latter are called Arens-Eells spaces and are denoted by \AE$(M)$.

The structure of Lipschitz-free spaces is a popular and very active research topic in Banach space theory. One direction of research has been the approximation properties of free spaces. Results involving the approximation property (AP) appear in \cites{godefroyozawa, hajeklancienpernecka, kalton}; results involving the bounded approximation property (BAP) appear in \cites{godefroykalton, godefroy:15, apandschd}; and metric approximation property (MAP) results can be found in in \cites{godefroykalton, adalet, godefroyozawa, douchakaufmann, cuthdoucha,fonfw,ps:15,smithtalim}. Other related results can be found in \cite{godefroy:20, godefroyappsequences}. More details are available in the introduction to \cite{smithtalim}.

We will mention several results regarding compact and proper metric spaces (the latter property means that every closed ball is compact). It is known that if $M$ is a sufficiently `thin' totally disconnected metric space then $\free{M}$ has the MAP. For example, if $M$ is a countable proper metric space then $\free{M}$ has the MAP \cite{adalet}. Also, as a corollary of \cite{weaver}*{Corollary 4.39} and \cite{cuthdoucha}*{Corollary 16}, $\free{M}$ has the MAP when $M$ is compact and uniformly disconnected, meaning that there exists $0 < a \leq 1$ such that for every distinct $p,q\in M$, we can find complementary clopen sets $C,D \subseteq M$ such that $p\in C, q\in D$ and $d(C,D) \geq a d(p,q)$, where $d(C,D) = \inf \{d(x,y) : x\in C, y\in D\}$. On the other hand, in \cite{godefroyozawa}, G.~Godefroy and N.~Ozawa constructed a compact convex subset $C$ of a separable Banach space such that $\free{C}$ fails the AP. Also, by a result in \cite{hajeklancienpernecka}, there exists a metric space $M$ homeomorphic to the Cantor space such that $\free{M}$ fails the AP. A characterisation of the BAP on free spaces over compact metric spaces, in terms of uniformly bounded sequences of Lipschitz `almost extension' operators, is given in \cite{godefroy:15}.

In \cite{godefroysurvey}, G.~Godefroy surveys various aspects of the theory of free spaces, including the lifting property for separable Banach spaces, approximation properties of free spaces, and norm attainment of Lipschitz functions and operators. In the last section he states a number of open problems, several of which concern approximation properties of free spaces.

For a compact and metrisable topological space $T$, we denote by $\cM^T$ the set of all metrics on $T$ compatible with its topology. We consider $\cM^T$ as a subset of $C(T^2)$, where the latter is equipped with the usual (supremum) norm. By \cite[Proposition 2.1]{stprbeforepub}, $\cM^T$ is a $G_\delta$ subset of $C(T^2)$, and hence completely metrisable. Thus $\cM^T$ is a Polish space. Let us define the sets
\begin{align*}
	\cA^{T,\lambda} &= \{d\in\cM^T \; :\; \free{T,d} \text{ has the } \lambda\text{-BAP}\}, \text{ for } \lambda \in [1,\infty), \\
	\text{and}\quad\cA^T_f &= \{d\in\cM^T \; :\; \free{T,d} \text{ fails the AP} \}.
\end{align*}

In \cite{godefroysurvey}*{Problem 6.6}, Godefroy considered the set $\cM^K$, where $K=2^\N$ is the Cantor space equipped with the usual  (product) topology. He asked about the topological nature of the set $\cA_f^K$ and, in particular, whether $\cA_f^K$ is a residual set in $\cM^K$. The latter question has a negative answer because $\cA^{K,1}$ is a residual ($F_{{\sigma\delta}}$) set in $\cM^K$ \cite{talimpublished}*{Theorem 1.1}. In addition, it was shown that the necessarily meager set $\cA^K_f$ is dense in $\cM^K$ \cite{talimpublished}*{Proposition 1.2}. This direction of investigation was continued in \cite{stprbeforepub}, where the authors showed, in particular, that whenever $T$ is compact and metrisable, $\cA^{T,1}$ is dense in $\cM^T$, and if $T$ is uncountable then $\cA_f$ is also dense in $\cM^T$. Moreover, they showed that $\cA^{T,1}$ is residual in $\cM^T$ when $T$ is zero-dimensional. In this paper we generalise the latter result to the case of strongly countable-dimensional spaces (for the definition please see \Cref{def:covdim,,def:strcount}). Our main result is the following. 

\begin{theorem}\label{thm:main}
	If $T$ is compact, metrisable and strongly countable-dimensional, then $\cA^{T,1}$ is residual in $\cM^T$. Consequently, $\cA^T_f$ is meager in $\cM^T$.
\end{theorem}

We note that one can also investigate the topological nature of the sets $\cA^{T,\lambda}$ and $\cA^T_f$ in a setting where the underlying space $T$ is not compact. Indeed, in \cite{stprbeforepub} the authors considered the (completely metrisable) space of all proper metrics on a given `properly metrisable' space $T$ (equivalently, a locally compact, separable and metrisable space $T$). Here the space of proper metrics is equipped with the topology of uniform convergence. We also mention that if $T$ is a completely metrisable topological space, then the space of all complete metrics on $T$, equipped with the topology of uniform convergence, is a Baire space \cite[Theorem 1.4]{ishiki2024spaces}. Therefore this provides another potential setting for investigation in this research direction.

The rest of the paper is organised as follows. In \Cref{sect_anc} we introduce the notation, the definitions related to the covering dimension, and some preliminary results. In \Cref{sec:findim} we prove a special case of \Cref{thm:main}, namely, that $\cA^{T,1}$ is residual in $\cM^T$ when $T$ is compact, metrisable, and has finite covering dimension. This result is used in \Cref{sec:strcount}, which is devoted to the proof of \Cref{thm:main}.

\section{Notation and preliminary results}\label{sect_anc}

For a metric space $(M,d)$ and $A,B\subseteq M$ we write
\begin{align*}
	d(A,B) = \inf_{x\in A,y\in B} d(x,y), \qquad d(x,A) = d(\{x\},A), \quad\text{ and }\quad \diam_d(M) = \sup_{x,y\in M} d(x,y).
\end{align*}
If $r>0$ we write $B^d_r(A) = \{x\in M : d(x,A) < r\}$, and for $x\in M$ we write $B^d_r(x)$ for $B^d_r(\{x\})$. If there is no ambiguity about the metric $d$ we write $B_r(A)$ and $B_r(x)$. We also write $\cl(A)$ for the closure of $A$ in $(M,d)$. For a real-valued Lipschitz function $f$ on $(M,d)$, we write $\Lip_d(f)$ for the (optimal) Lipschitz constant of $f$ with respect to $d$. The Lipschitz space and free space over $(M,d)$ are denoted by $\Lipo{M,d}$ and $\free{M,d}$, respectively. If there is no ambiguity about the metric $d$ we omit it from the notation. For a bounded real-valued function $g$ defined on a set $U$, we write $\nm g _\infty = \sup_{x\in U} |g(x)|$. If $A\subseteq U$ is a subset then we write $g\restrict_A$ for the restriction of $g$ to $A$. For a subset $A\subseteq M$ and $\epsilon > 0$, we say that $A$ is \emph{$\epsilon$-dense} in $M$ if $d(x,A) \leq \epsilon$ for every $x\in M$. We call an operator $E\colon C(A)\to C(M)$ an \emph{extension operator} if $E(f)\restrict _A = f$ for all $f\in C(A)$. If $X$ is a Banach space then $B_X$ denotes the closed unit ball of $X$. 

We will first give a useful criterion, due to Godefroy, for the $\lambda$-BAP of free spaces over compact metric spaces in terms of `almost-extension' operators \cite{godefroysurvey}*{Theorem 3.2} (also \cite{godefroy:15}*{Theorem 1}).

\begin{theorem}\label{thm:godefroycriterion}
	Let $M$ be a compact metric space and $(M_n)_n$ be a sequence of finite $\epsilon_n$-dense subsets of $M$, where $(\epsilon_n)_n$ is a sequence of positive real numbers such that $\lim_{n\to\infty} \epsilon_n = 0$. Then $\free{M}$ has the $\lambda$-BAP if and only if there is a sequence $(T_n)_n$ of operators $T_n\colon\Lipo{M_n}\to\Lipo{M}$ such that $\nm{T_n} \leq \lambda$ for all $n$ and $$\lim_{n\to\infty} \sup_{f\in B_{\Lipo{M_n}}} \nm{T_n(f)|_{M_n}-f}_\infty = 0.$$
\end{theorem}

We will need the following theorem, whose proof can be found in e.g.~\cite{lindenstrausstzafriri}*{Theorem 1.e.15}.

\begin{theorem}[Grothendieck]\label{thm:mapgrothendieck}
	If $X$ is a separable dual Banach space with the AP then $X$ has the MAP.
\end{theorem}

Let $T$ be a compact and metrisable space. For $d\in\cM^T$, recall that the metric space $(T,d)$ is called \emph{purely 1-unrectifiable} if there is no pair $(A,h)$, where $A\subseteq \R$ is a compact set of positive Lebesgue measure, and $h\colon A\to (T,d)$ is a bilipschitz map (see the end of the paragaph before \cite[Theorem A]{purely1unrect}). Define the set $$\cP^T = \{d\in\cM^T : (T,d) \text{ is purely 1-unrectifiable}\}.$$ The next proposition is similar to \cite[Proposition 3.6]{talimpublished}. The only difference in the proofs is the substitution of the space $K$ (in \cite{talimpublished}) with $T$, and the addition of the density argument at the end of the next proof. We include the proof of this result for completeness.

\begin{proposition}\label{prop:p1uresidual}
	The set $\cP^T$ is dense and $G_\delta$ in $\cM^T$.
\end{proposition}	
\begin{proof}
	We first prove that $\cP^T$ is a $G_\delta$ set in $\cM^T$. We denote the Lebesgue measure of a measurable set $A\subseteq \R$ by $\Delta(A)$. For $m,k\in\N$ define
	\begin{align*}
		V_{m,k} = \{& d\in \cM^T \; :\; \text{there exists a compact } K\subseteq T \text{ and } h\colon K \to \R, \text{ such that } \\& m^{-1} d(x,y) \leq |h(x)-h(y)| \leq md(x,y) \text{ for all } x,y\in K, \text{ and } \Delta(h(K)) \geq k^{-1} \},
	\end{align*}
	We will prove that $V_{m,k}$ is closed. Let $(d_n)_{n\in\N} \subseteq V_{m,k}$ converge uniformly to $d\in\cM^T$, and let $K_n$ and $h_n\colon K_n\to\R$ be the compact set and bilipschitz function, respectively, associated to each $d_n$. The space $2^T$ of all closed subsets of $T$, equipped with the Vietoris topology, is compact and metrisable by \cite[Problem 3.12.27 and Problem 4.5.23 (e)]{engelking}. Therefore, by taking a subsequence, we can assume that the $K_n$ converge to a compact set $K\subseteq T$ in the Vietoris topology. We extend each function $h_n$ to a function (again denoted by $h_n$) on $T$ by \cite{weaver}*{Theorem 1.33}, while preserving its Lipschitz constant with respect to $d_n$. We can assume, by translation, that $0\in h_n(T)$ for each $n$. As the diameters of $(T,d_n)$ are uniformly bounded, the functions $h_n$ are uniformly bounded. 
	
	We will show that they are also $d$-equicontinuous. Pick $\epsilon > 0$, and pick $n\in\N$ such that $\nm{d_i-d}_\infty < \frac{\epsilon}{2m}$ for $i\geq n$. Then for $i\geq n$ and $x,y\in T$ such that $d(x,y) < \frac{\epsilon}{2m}$, $$|h_i(x) - h_i(y)| \leq md_i(x,y) \leq m\left(\frac{\epsilon}{2m} + d(x,y)\right) < \epsilon.$$ Therefore $(h_n)_n$ is $d$-equicontiuous. By the Arzel\` a-Ascoli theorem, there exists $h\in C(T)$ and a subsequence $(h_{n_i})_i$ of $(h_n)_n$ such that $h_{n_i} \to h$ uniformly. For $x,y\in T$, $$|h(x)- h(y)| = \lim_{i\to\infty} |h_{n_i}(x)- h_{n_i}(y)| \leq \lim_{i\to\infty} m d_{n_i}(x,y) = md(x,y).$$ Hence $\Lip_d (h) \leq m$. Now let $x,y\in K$ be arbitrary and let $(x_i)_i, (y_i)_i \subseteq T$ converge to $x$ and $y$, respectively, and be such that $x_i,y_i\in K_{n_i}$ for each $i\in\N$. Let $\epsilon > 0$ and pick $i$ such that $$d(x,x_i), d(y, y_i), \nm{h - h_{n_i}}_\infty, \nm{d-d_{n_i}}_\infty < \epsilon.$$ Then
	\begin{align*}
		|h(x)-h(y)| &\geq |h(x_i) - h(y_i)| - |h(x) - h(x_i)| - |h(y) - h(y_i)| \\&\geq |h(x_i) - h(y_i)| - 2m\epsilon \geq |h_{n_i}(x_i) - h_{n_i}(y_i)| - (2+2m)\epsilon \\&\geq m^{-1} d_{n_i}(x_i,y_i) - (2+2m)\epsilon \geq m^{-1} d(x_i,y_i) - (2+2m+m^{-1})\epsilon \\&\geq m^{-1} (d(x,y)-d(x,x_i)-d(y,y_i)) - (2+2m+m^{-1})\epsilon \\&\geq m^{-1} d(x,y) - (2+2m+3m^{-1})\epsilon.
	\end{align*}
	As $\epsilon$ was arbitrary, we get $|h(x)-h(y)| \geq m^{-1}d(x,y)$. This means that $h$ satisfies the bilipschitz condition in the definition of $V_{m,k}$ on the set $K$.
	
	To show that $\Delta(h(K)) \geq k^{-1}$, pick $\epsilon > 0$ and set $U = B_\epsilon^{|\cdot|}(h(K))$. We have that $h^{-1}(U)$ is an open set in $T$ containing $K$. Choose $i$ such that $K_{n_i} \subseteq h^{-1}(U)$ and $\nm{h-h_{n_i}}_\infty < \epsilon$. Then $$h_{n_i}(K_{n_i}) \subseteq B_\epsilon^{|\cdot|}(h(K_{n_i})) \subseteq B_\epsilon^{|\cdot|}(U) \subseteq B_{2\epsilon}^{|\cdot|}(h(K)).$$ Hence $\Delta(B_{2\epsilon}^{|\cdot|}(h(K))) \geq \Delta(h_{n_i}(K_{n_i})) \geq k^{-1}$. As $\epsilon$ was arbitrary, we get  $\Delta(h(K)) \geq k^{-1}$. Therefore $d\in V_{m,k}$ and so $V_{m,k}$ is closed. As $\cM^T\setminus\cP^T = \bigcup_{m,k\in\N} V_{m,k}$, we have that $\cM^T\setminus\cP^T$ is $F_\sigma$, and so $\cP^T$ is $G_\delta$. To show that $\cP^T$ is dense, pick $d\in\cM^T$ and $\epsilon > 0$. We can find some $\delta > 0$ such that $\nm{d^{1-\delta} - d}_\infty < \epsilon$. By the second sentence after the proof of \cite[Theorem 4.6]{purely1unrect}, $d^{1-\delta} \in\cP^T$. Therefore $\cP^T$ is dense in $\cM^T$.
\end{proof}

The next lemma will be needed in the proof of \Cref{thm:main} in \Cref{sec:strcount}.

\begin{lemma}\label{lm:residualitysubset}
	If $C\subseteq T$ is a closed subset and $\cB\subseteq \cM^C$ is residual in $\cM^C$, then the set $\cU = \{d\in \cM^T : d\restrict _{C^2}\in \cB\}$ is residual in $\cM^T$.
\end{lemma}
\begin{proof}
	Let $(\cB_n)_n$ be a sequence of dense open sets in $\cM^C$ such that $\bigcap _{n=1}^\infty \cB_n \subseteq \cB$. Define $\cU_n = \{d\in \cM^T : d\restrict _{C^2}\in \cB_n\}$ for each $n\in\N$. For $n\in\N$, it is easy to see that  $\cU_n$ is open. To show that it is dense, let $d\in\cM^T$ and $\epsilon > 0$. As $\cB_n$ is dense in $\cM^C$, we can find $e\in \cB_n$ such that $\nm{e-d\restrict _{C^2}}_\infty < \epsilon$. By \cite[Theorem 1.1]{ishiki2020interpolation}, we can find an extension $\bar e\in\cM^T$ of $e$ such that $\nm{\bar e - d}_\infty < \epsilon$. Therefore $\bar e \in \cU_n$ and so $\cU_n$ is dense in $\cM^T$. Finally, $\bigcap_{n=1}^\infty \cU_n \subseteq \cU$ and $\cU$ is residual.
\end{proof}

We will next define the covering dimension of normal topological spaces by following \cite[Section 1.6]{engelking1995theory}. Let $S$ be a normal topological space. Suppose that $\cV$ a family of subsets of $S$. We define the \emph{order} of $\cV$ to be the largest integer $n$ such that there are $n+1$ sets in $\cV$ with nonempty intersection. If there is no largest such integer, then we say that $\cV$ is of infinite order.

\begin{definition}[\text{cf.~\cite[Definition 1.6.7]{engelking1995theory}}]\label{def:covdim}
	The \emph{covering dimension} of $S$, denoted by $\dim S$, is defined by the following conditions:
	\begin{enumerate}[label=(\roman*)]
		\item $\dim S \leq n$ if every finite open cover of $S$ has a finite refinement of order at most $n$,
		\item $\dim S = n$ if $\dim S\leq n$ and $\dim S \leq n-1$ does not hold,
		\item $\dim S = \infty$ if $\dim S \leq n$ does not hold for every integer $n$.
	\end{enumerate}
\end{definition}

Recall that if $S$ is separable and metrisable, and $P\subseteq S$ is a subspace, then $\dim P\leq \dim S$ by \cite[Theorem 1.1.2 and Theorem 1.7.7]{engelking1995theory}. We next introduce definitions related to infinite-dimensional spaces.

\begin{definition}[\text{cf.~\cite[Definition 5.5.1]{engelking1995theory}}]
	The space $S$ is called \emph{locally finite dimensional} if every point $x\in S$ has a normal neighbourhood $U$ such that $\dim U < \infty$. 
\end{definition}

\begin{definition}[\text{cf.~\cite[Definition 5.1.4]{engelking1995theory}}]\label{def:strcount}
	The space $S$ is called \emph{strongly countable-dimensional} if $S = \bigcup_{n=1}^\infty S_n$ where each $S_n$ is a closed finite-dimensional subset of $S$.
\end{definition}

Using transfinite recursion, we will define the transfinite dimensional kernel and transfinite kernel dimension of $S$, following \cite[Section 7.3]{engelking1995theory}. For an ordinal $\alpha$, let $\lambda(\alpha)$ and $l(\alpha)$ be the unique limit and finite ordinal, respectively, so that $\alpha = \lambda(\alpha) + l(\alpha)$. We define a decreasing transfinite sequence (with respect to inclusion) $(S_\alpha)_\alpha$, and a transfinite sequence $D_\alpha(S)$ of subsets of $S$. Define $$D_n(S) = \bigcup \{U\subseteq S : U \text{ is open and }\dim \cl(U) \leq n\}.$$ Then define $$S_\alpha = S\setminus \bigcup_{\beta < \alpha} D_\beta(S), \text{ and } D_\alpha(S) = D_{l(\alpha)}(S_{\lambda(\alpha)}).$$ 

By \cite[Proposition 7.3.2]{engelking1995theory}, $S_\alpha$ is closed in $S$ for every $\alpha$. Since $(S_\alpha)_\alpha$ is decreasing (with respect to inclusion), there exists an ordinal $\alpha$ such that $S_{\beta} = S_\gamma$ for all $\beta, \gamma\geq \alpha$. If $\alpha$ is the least such ordinal, then $S_\alpha$ is called the \emph{transfinite dimensional kernel} of $S$, and it is denoted by $k(S)$. We define the \emph{transfinite kernel dimension} of $S$, $\trker S$, by setting $\trker S = \alpha$ if $S_\alpha = \emptyset$, and $\trker S = \infty$ otherwise (here $\infty$ is a symbol distinct from all ordinals). We end this section with the following lemma, which will be used in the proof of the main result in \Cref{sec:strcount}.
\begin{lemma}\label{lm:kernelemptyinter}
	If $S$ is metrisable and $k(S) = \emptyset$, then $\inter(S_\omega) = \emptyset$.
\end{lemma}
\begin{proof}
	Suppose that $U$ is a nonempty open set in $S$ such that $U\subseteq S_\omega$. Then $D_n(U) = \emptyset$ for all $n\in\N$, and so $U_n = U$ for all $n\in\N$. It is then easy to see that $U_\alpha = U$ for all ordinals $\alpha$. But by \cite[Lemma 7.3.7 (ii)]{engelking1995theory}, $U_{\trker S} = U\cap S_{\trker S} = \emptyset$, a contradiction.
\end{proof}

\section{The finite dimensional case}\label{sec:findim}

	The main result of this section is the following proposition.
	
	\begin{proposition}\label{prop:findimres}
		Let $T$ be a compact, metrisable and finite-dimensional space. Then $\cA^{T,1}$ is residual in $\cM^T$.
	\end{proposition}
	
	We will need the following lemma.
	\begin{lemma}\label{lm:nicecoverconstruct}
		Suppose that $T$ is a compact, metrisable and finite-dimensional space. Let $a_0\in T$, $d\in\cM^T$ and $\epsilon > 0$. Then there exists an open cover $(U_i)_{i=0}^n$ of $T$ of order at most $\dim T$ and a set $A = \{a_0,\ldots,a_n\}\subseteq T$ such that 
		\begin{enumerate}[label=(\roman*)]
			\item $a_i\in U_j$ if and only if $i=j$, \label{eq:aiinujiff}
			\item $U_i\subseteq B_{\frac{\epsilon}{2}}(a_i)$ for each $i$, \label{eq:duileqepsilon}
			\item $d(a_i,a_j) > \frac{\epsilon}{3}$ whenever $i\not=j$. \label{eq:aiajfar}
		\end{enumerate}
	\end{lemma}
	\begin{proof}
		Assume the premise of the lemma and set $r=\dim T$. By compactness, we can find a finite open cover $\cV'$ of $T$ such that $\diam_d(V') < \frac{\epsilon}{6}$ for each $V'\in\cV'$. 
		We can then find a finite open cover $\cV$ of $T$ of order at most $r$ which is a refinement of $\cV'$. Therefore,
		\begin{align}\label{eq:diamVleq}
			\diam_d (V) < \frac{\epsilon}{6} \text{ for each } V\in\cV.
		\end{align}
		By omitting some members of $\cV$, we can assume that $\bigcup (\cV\setminus\{V\}) \not= T$ for any $V\in\cV$. Pick $V_0\in\cV$ such that $a_0\in V_0$. If $a_0\in V$ for some $V\in\cV$ with $V\not=V_0$, then we shrink $V$ so that $a_0\not\in V$ and $\cV$ still covers $T$. We can therefore pick a point $a_V\in T\setminus \bigcup (\cV\setminus\{V\})$ for each $V\in \cV$ (we pick $a_0$ for $V_0$). Now pick a subset $\cW\subseteq \cV$ such that the set $\{a_W : W\in\cW\}$ contains $a_0$, and satisfies $d(a_W,a_{W'}) > \frac{\epsilon}{3}$ whenever $W\not=W'$, and for each $V\in\cV$ there exists $W\in\cW$ such that $d(a_W,a_V) \leq \frac{\epsilon}{3}$. Let $f\colon \cV\to\cW$ be any map satisfying $f(W) = W$ for all $W\in\cW$, and 
		\begin{align}\label{eq:davafvleq}
			d(a_V, a_{f(V)}) \leq \frac{\epsilon}{3} \text{ for all } V\in \cV.
		\end{align}
		Set $U_W = \bigcup f^{-1}(W)$ for $W\in\cW$, and $\cU = \{U_W : W\in \cW\}$.  Clearly $\cU$ is a cover of $T$. We have $a_W\in W\subseteq U_W$ for $W\in\cW$ since $W\in f^{-1}(W)$. Also, if $W'\in\cW\setminus\{W\}$, then $W'\not\in f^{-1}(W)$, since $d(a_{W'},a_W) > \frac{\epsilon}{3}$. This implies that $a_{W'}\not\in U_W$. If $x\in U_W$ then $x\in V$ for some $V\in f^{-1}(W)$, and then $$d(x,a_W) \leq d(x,a_V) + d(a_V,a_W) < \frac{\epsilon}{6} + \frac{\epsilon}{3} = \frac{\epsilon}{2},$$ by \eqref{eq:diamVleq} and \eqref{eq:davafvleq}. This shows that $U_W \subseteq B^d_{\frac{\epsilon}{2}}(a_W)$. We have proven that the cover $\cU = \{U_W : W\in\cW\}$ and the points $\{a_W : W\in \cW\}$ satisfy \ref{eq:aiinujiff} - \ref{eq:aiajfar}.
		
		We finally show that $\cU$ has order at most $r$. Suppose $x\in U_{W_1}\cap\ldots\cap U_{W_k}$, for some distinct $W_1,\ldots,W_k\in \cW$. As $f^{-1}(W)$ and $f^{-1}(W')$ are disjoint whenever $W,W'\in \cW$, $W\not=W'$, we have $x\in V_1\cap \ldots\cap V_k$ for some distinct $V_1,\ldots,V_k\in \cV$. As the order of $\cV$ is at most $r$, we have $k\leq r+1$.		
	\end{proof}
	
	The following proposition is used to define extension operators with a controlled norm, which extend Lipschitz functions from finite subsets of $T$ to $T$. The first part of this proposition is similar to \cite[Lemma 2.6]{stprbeforepub}.
	
	\begin{proposition}\label{prop:extmet}
		Let $T$ be a compact and metrisable space, $d\in\cM^T$ and $\epsilon \in (0,1)$. Assume that $(U_i)_{i=0}^n$ is an open cover of $T$ of order at most $r$ and $A = \{a_0,\ldots,a_n\}\subseteq T$ is a finite set satisfying \ref{eq:aiinujiff} - \ref{eq:aiajfar}.Then there exists $\bar d \in \cM^T$ and an extension operator $E\colon \Lipo{A, \bar d}\to \Lipo{T,\bar d}$, such that $\bar d$ extends $d\restrict _{A^2}$, $\nm{d - \bar d}_\infty < 4\epsilon$ and $\nm E = 1$. Moreover, for every $e\in \cM^T$ such that $\nm{\bar d - e}_\infty \leq \frac{\epsilon}{12(r+1)}$, there exists an extension operator $G\colon\Lipo{A,e}\to\Lipo{T,e}$ such that $\nm G \leq 88(r+1)(2r+3)$.
	\end{proposition}
	\begin{proof}
		Assume the premise of the proposition, and suppose $a_0$ is the base point of $A$ and $T$. Let $(\lambda_i)_i$ be the usual partition of unity associated to the cover $(U_i)_i$ and the metric $d$, i.e.~$\lambda_i(x) = \frac{d(x,U_i^c)}{\sum_j d(x,U_j^c)}$ for each $i=0,\ldots,n$. Define $F\colon C(A)\to C(T)$ by $F(f) = \sum_{i=0}^n f(a_i)\lambda_i$ and define the pseudometric $\tildem d\colon T^2\to [0,\infty)$ by 
		\begin{align}\label{eq:deftildemd}
			\tildem d(x,y) = \sup_{f\in B_{\Lipo{A,d}}} |F(f)(x)-F(f)(y)|.
		\end{align}
	
		We estimate $\nm{\tildem d - d}_\infty$. Pick $x,y\in T$, and suppose $x\in U_i$ and $y\in U_j$ for some $i,j$. If $f\in\Lipo{A,d}$ then $F(f)(x)$ is a convex combination of those $f(a_k)$ for which $x\in U_k$, $k=0,\ldots,n$. Note that for such $k$, $$d(a_i,a_k)\leq d(a_i,x) + d(x,a_k) < \epsilon,$$ by \ref{eq:duileqepsilon}. Therefore, $$|F(f)(x) - f(a_i)| \leq \max \{|f(a_k) - f(a_i)| \: : \: x\in U_k,\:  k=0,\ldots,n\} < \epsilon.$$ Similarly, $|F(f)(y) - f(a_j)| < \epsilon$. By the triangle inequality, $$\tildem d(x,y) = \sup_{f\in B_{\Lipo{A,d}}} |F(f)(x) - F(f)(y)| < |f(a_i) - f(a_j)| + 2\epsilon \leq d(a_i,a_j) + 2\epsilon < d(x,y) + 3\epsilon,$$ where the last inequality holds since $d(x,a_i), d(y,a_j) < \frac{\epsilon}{2}$ by \ref{eq:duileqepsilon}. On the other hand, let us pick $f\in B_{\Lipo{A,d}}$ such that $f(a_i)-f(a_j) = d(a_i,a_j)$. Then $$\tildem d(x,y) \geq |F(f)(x) - F(f)(y)| > |f(a_i) - f(a_j)| - 2\epsilon = d(a_i,a_j) - 2\epsilon > d(x,y) - 3\epsilon.$$ Therefore $\nm{\tildem d - d}_\infty < 3\epsilon$. Now define the pseudometric $\hat d\colon T^2\to [0,\infty)$ by $$\hat d(x,y) = \min(d(x,y), d(x,A) + d(A,y)).$$ Since $A$ is $\frac{\epsilon}{2}$-dense in $T$ with respect to $d$ by \ref{eq:duileqepsilon}, we have $\nm {\hat d} _\infty \leq \epsilon$. We define $\bar d = \tildem d + \hat d$. It is clear that $\bar d$ is a metric on $T$ which extends $d\restrict _{A^2}$ and that 
		\begin{align}\label{eq:dbardclose}
			\nm{d-\bar d}_\infty < 4\epsilon.
		\end{align} 
		To show that $\bar d\in\cM^T$, by compactness it suffices to show that $\bar d(x_k,x)\to 0$ whenever $(x_k)_k\subseteq T$, $x\in T$ and $d(x_k,x)\to 0$. As $\hat d\leq d$, it suffices to show $\tildem d(x_k,x)\to 0$. Suppose, for a contradiction, that there exists $(x_k)_k\subseteq T$, $x\in T$ such that $d(x_k,x)\to 0$ and $\tildem d(x_k,x) > \eta$ for some $\eta > 0$. Then there exists a sequence $(f_k)_k\subseteq B_{\Lipo{A,d}}$ such that $|F(f_k)(x_k)-F(f_k)(x)| > \eta$ for all $k$. As the $f_k$ are uniformly bounded, we can find a subsequence $(f_{k_m})_m$ converging uniformly to some $f\in B_{\Lipo{A,d}}$. It is not hard to see that then $F(f_{k_m}) \to F(f)$ uniformly as well. Now it is clear that this leads to a contradiction. Hence $\bar d\in \cM^T$. We define $E\colon\Lipo{A, \bar d}\to\Lipo{T, \bar d}$ by $E(f) = F(f)$, and note that $E$ is an extension operator such that $\nm E = 1$, due to \eqref{eq:deftildemd}, the inequality $\bar d \geq \tildem d$ and the fact that $\bar d$ agrees with $d$ on $A$.
		
		To proceed further, we will need several estimates. For $i=1,\ldots,n$, define $h_i\colon A\to\R$ by $h_i(a_i) = 1$ and $h_i(a_j) = 0$ whenever $j\not=i$. We have 
		\begin{align}\label{eq:liplambdaiest}
			\Lip_{\bar d}(\lambda_i) = \Lip_{\bar d}(E(h_i)) \leq \Lip_{\bar d}(h_i) = \frac{1}{\min_{j\not=i} \bar d(a_j,a_i)} = \frac{1}{\min_{j\not=i} d(a_j,a_i)} \leq \frac{3}{\epsilon},
		\end{align}
		by \ref{eq:aiajfar} and the fact that $\bar d$ agrees with $d$ on $A$. Pick $x\in T$ and $i$ such that $x\in U_i$, and suppose $y\in U_i^c$ is such that $\bar d(x,U_i^c) = \bar d(x,y)$. Then $$\lambda_i(x) = |\lambda_i(x) - \lambda_i(y)| \leq \frac{3}{\epsilon} \bar d(x,y) = \frac{3}{\epsilon} \bar d(x,U_i^c),$$ by \eqref{eq:liplambdaiest}. Therefore,
		\begin{align}\label{eq:forxinTsumdest}
			\sum_i \bar d(x, U_i^c) \geq \frac{\epsilon}{3} \sum_i \lambda_i(x) = \frac{\epsilon}{3}.
		\end{align}
		Now pick a metric $e\in \cM^T$ satisfying
		\begin{align}\label{eq:ebardclose}
			\nm{e - \bar d}_\infty \leq \frac{\epsilon}{12(r+1)},
		\end{align} 
		and define the partition of unity $(\mu_i)_{i=0}^n$ subordinate to $(U_i)_{i=0}^n$ by $\mu_i(x) = \frac{e(x,U_i^c)}{\sum_j e(x,U_j^c)}$. Let us estimate $\Lip_{e}(\mu_i)$. Note first that $\Lip_{e} (e(\cdot, S)) \leq 1$ for every closed non-empty subset $S\subseteq T$. Pick $x,y\in T$. We have 
		\begin{align*}
			&\;|\mu_i(x) - \mu_i(y)| = \left|\frac{e(x,U_i^c)}{\sum_j e(x,U_j^c)} - \frac{e(y,U_i^c)}{\sum_j e(y,U_j^c)}\right| \\=\;& \frac{|e(x,U_i^c)\sum_j e(y,U_j^c) - e(x,U_i^c)\sum_j e(x,U_j^c) + e(x,U_i^c)\sum_j e(x,U_j^c) - e(y,U_i^c)\sum_j e(x,U_j^c)|}{\sum_j e(x,U_j^c) \sum_j e(y,U_j^c)} \\ \leq \; & \frac{e(x,U_i^c) 2(r+1) e(x,y) + e(x,y)\sum_j e(x,U_j^c)}{\sum_j e(x,U_j^c) \sum_j e(y,U_j^c)} \leq  \frac{(2r+3)e(x,y)}{\sum_j e(y,U_j^c)} \\ \leq\;& \frac{(2r+3) e(x,y)}{\sum_j \bar d(y,U_j^c) - \sum_j |\bar d(y,U_j^c) - e(y,U_j^c)|} \leq \frac{(2r+3) e(x,y)}{\frac{\epsilon}{3} - \frac{\epsilon}{12}} \tag*{by \eqref{eq:forxinTsumdest} and \eqref{eq:ebardclose}} \\ \leq\; & \frac{4(2r+3) e(x,y)}{\epsilon}.
		\end{align*}
		Therefore,
		\begin{align}\label{eq:lipmuiest}
			\Lip_{e} (\mu_i) \leq \frac{4(2r+3)}{\epsilon}.
		\end{align}
	
		Finally, define $G\colon\Lipo{A,e}\to\Lipo{T,e}$ by $G(f) = \sum_i f(a_i)\mu_i$. To estimate $\nm G$, suppose $f\in B_{\Lipo{A,e}}$ and $x,y\in T$. Pick $i_0$ such that $x\in U_{i_0}$. Assume first that $e(x,y) < \epsilon$. If $i$ is such that $x\in U_i$ then using \eqref{eq:ebardclose}, \eqref{eq:dbardclose} and \ref{eq:duileqepsilon} we get 
		\begin{align}\label{eq:eaiai0firstest}
			e(a_i,a_{i_0}) < \frac{\epsilon}{2} + \bar d(a_i,a_{i_0}) < \frac{9}{2}\epsilon + d(a_i,a_{i_0}) \leq \frac{9}{2}\epsilon + d(a_i,x) + d(x,a_{i_0}) < \frac{11}{2}\epsilon.
		\end{align}
		If $i$ is such that $y\in U_i$ then similarly $$e(a_i,a_{i_0}) < \frac{9}{2}\epsilon + d(a_i,a_{i_0}) \leq \frac{9}{2}\epsilon + d(a_{i_0},x) + d(x,y) + d(y,a_i) < \frac{11}{2}\epsilon + d(x,y) < 10\epsilon + e(x,y) < 11\epsilon.$$ Therefore,
		\begin{align}\label{eq:eaiai0small}
			e(a_i,a_{i_0}) < 11\epsilon \text{ whenever } i \text{ is such that } U_i\cap \{x,y\}\not=\emptyset.
		\end{align} 
		Keeping in mind that $\sum_{i} \mu_i(x) - \mu_i(y) = 0$, we now estimate
		\begin{align}
			|G(f)(x) - G(f)(y)| &= \left|\sum_i f(a_i)(\mu_i(x) - \mu_i(y))\right| = \left|\sum_i (f(a_i)-f(a_{i_0}))(\mu_i(x) - \mu_i(y))\right| \nonumber \\&\leq \sum_i |f(a_i)-f(a_{i_0})||\mu_i(x) - \mu_i(y)| \nonumber\\&\leq 22(r+1)\epsilon \frac{4(2r+3)}{\epsilon} e(x,y) \tag*{by \eqref{eq:eaiai0small} and \eqref{eq:lipmuiest}} \\&= 88(r+1)(2r+3) e(x,y), \label{eq:estGfxyclose}
		\end{align}
			
		Assume now $e(x,y)\geq \epsilon$. If $i$ is such that $x\in U_i$ then $e(a_i,a _{i_0}) \leq \frac{11}{2}\epsilon$ by \eqref{eq:eaiai0firstest}. Keeping in mind that $\sum_{i} \mu_i(x) - \lambda_i(x) = 0$, we first estimate
		\begin{align*}
			|G(f)(x) - E(f)(x)| &= \left|\sum_i f(a_i)(\mu_i(x) - \lambda_i(x))\right| \\&= \left|\sum_i (f(a_i) - f(a_{i_0}))(\mu_i(x) - \lambda_i(x))\right|< \frac{11}{2}\epsilon (r+1).
		\end{align*}
		Similarly, $|G(f)(y) - E(f)(y)| < \frac{11}{2}\epsilon (r+1)$. Suppose $i,j\in\{0,\ldots,n\}$ and $i\not=j$. Using \eqref{eq:ebardclose}, we estimate
		\begin{align*}
			\frac{|f(a_i)-f(a_j)|}{\bar d(a_i,a_j)} &\leq \frac{e(a_i,a_j)}{\bar d(a_i,a_j)} \leq \frac{\bar d(a_i,a_j) + \frac{\epsilon}{12(r+1)}}{\bar d(a_i,a_j)} \\&= 1 + \frac{\epsilon}{12(r+1)\bar d(a_i,a_j)} \leq 1 + \frac{1}{4(r+1)} \tag*{by  \ref{eq:aiajfar}.}
		\end{align*}
		Therefore $\Lip_{\bar d}(f) \leq 1 + \frac{1}{4(r+1)} < 2$. We now have
		\begin{align*}
			|G(f)(x) - G(f)(y)| &= |G(f)(x) - E(f)(x)| + |E(f)(x) - E(f)(y)| + |E(f)(y) - G(f)(y)| \\&< 11\epsilon (r+1) + |E(f)(x) - E(f)(y)| \leq 11\epsilon (r+1) + 2\bar d(x,y) \\&< 11\epsilon (r+1) + 2(\epsilon + e(x,y)) \leq (11r+15)e(x,y) \tag*{by \eqref{eq:ebardclose}.}
		\end{align*}
		Hence, by \eqref{eq:estGfxyclose} we have $\Lip_e(G(f))\leq 88(r+1)(2r+3)$ and so $\nm G \leq 88(r+1)(2r+3)$.
	\end{proof}

	\begin{proof}[Proof of \Cref{prop:findimres}.]	
	Put $r=\dim T$ and pick a base point $a_0\in T$. For $n\in\N$, define the set
	\begin{align*}
		\cB_n = \{ d\in \cM^T : \: &\text{ there exists a finite subset } A\subseteq T \text{ containing } a_0 \text{ and an extension }\\& \text{ operator }  E\colon \Lipo{A,d}\to\Lipo{T,d}, \text{ such that } A \text{ is } \frac{1}{n}\text{-dense in } T \\& \text{ with respect to } d \text{ and }  \nm E \leq 88(r+1)(2r+3) \}.
	\end{align*}

	We will show that for every $n\in\N$, $\inter(\cB_n)$ is dense in $\cM^T$. Pick $n\in\N$, $d\in\cM^T$ and $\nu > 0$. Put $\epsilon = \min(\frac{\nu}{4}, \frac{1}{10n})$. By \Cref{lm:nicecoverconstruct}, there exists an open cover $(U_i)_{i=0}^k$ of $T$ of order at most $r$ and a finite set $A = \{a_0,\ldots,a_k\} \subseteq T$ satisfying \ref{eq:aiinujiff} - \ref{eq:aiajfar}. Let $\bar d\in\cM^T$ be the metric obtained from \Cref{prop:extmet} applied to $d$, $\epsilon$, the cover $(U_i)_i$ and the finite set $A$. We have $\nm{\bar d - d}_\infty \leq 4\epsilon \leq \nu$. Suppose $e\in\cM^T$ satisfies $\nm{\bar d - e}_\infty \leq \frac{\epsilon}{12(r+1)}$. Then there exists an extension operator $G\colon\Lipo{A,e}\to\Lipo{T,e}$ satisfying $\nm G \leq 88(r+1)(2r+3)$. Moreover, $$\diam_e (U_i) < \diam_{\bar d} (U_i) + \epsilon\leq \diam_d (U_i) + 5\epsilon \leq 6\epsilon \leq \frac{1}{n},$$ by \ref{eq:duileqepsilon}. This shows that $A$ is $\frac{1}{n}$-dense in $T$ with respect to $e$. Hence $e\in \cB_n$, which shows that $\bar d\in \inter(\cB_n)$. Therefore $\inter(\cB_n)$ is dense in $\cM^T$.
	
	We define the set $\cB = \cP^T \cap \bigcap_{n=1}^\infty \cB_n$, and note that $\cB$ is residual by \Cref{prop:p1uresidual}. Pick $d\in\cB$. Then there exists sequences $(A_n)_n$ and $(E_n)_n$, such that $x_0\in A_n\subseteq T$, $A_n$ is finite and $\frac{1}{n}$-dense in $T$, and $\nm{E_n} \leq 88(r+1)(2r+3)$ for each $n$. By \Cref{thm:godefroycriterion}, $\free{T,d}$ has the $88(r+1)(2r+3)$-BAP, and since $d\in\cP^T$, $\free{T,d}$ is a dual space by \cite[Theorem B]{purely1unrect}. We note that $\free{T,d}$ is separable, since $\free{T,d}$ is the closed linear span of the image of the canonical (metric) isometry $(T,d)\to \free{T,d}$, $x\to\delta_x$. Hence $\free{T,d}$ has the MAP by \Cref{thm:mapgrothendieck}. Therefore $\cB\subseteq \cA^{T,1}$. As $\cB$ is residual in $\cM^T$, $\cA^{T,1}$ is residual in $\cM^T$ as well.
\end{proof}

\section{Strongly countable-dimensional spaces}\label{sec:strcount}

This section is devoted to the proof of \Cref{thm:main}. We will first need the following construction and proposition. Let $T$ be a compact metrisable topological space, and suppose $K\subseteq T$ is a closed subset such that $\dim K < \infty$. Let $(C_n)_n$ be an increasing sequence (with respect to inclusion) of closed subsets of $T\setminus K$ such that 
\begin{align}\label{eq:Cnchoice}
	T\setminus K = \bigcup_{n=1}^\infty \inter(C_n).
\end{align}
Let us fix a base point $a_0\in K$ for all Lipschitz spaces considered in this paragraph and \Cref{prop:cupRndense}. Set 
\begin{align}\label{eq:defGamma}
	\Gamma = 88(\dim K + 1)(2\dim K + 3)
\end{align}
for convenience, and for every $n,m\in\N$, $m\geq n$, define the set 
\begin{align}\label{eq:defRn}
	\cR_{n,m} = \{d\in\cM^T : &\text{ there exists a finite set } A\subseteq K \text{ containing } a_0, \text{ and a dual operator } \nonumber \\&\; H\colon\Lipo{C_m\cup A,d}\to\Lipo{T,d} \text{ satisfying } \nm H \leq (150\dim K + 152)(\Gamma + 1) \nonumber\\& \text{ and } H(f)\restrict _{C_n\cup A} = f\restrict _{C_n\cup A} \text{ for every } f\in\Lipo{C_m\cup A,d}\}.
\end{align}

\begin{proposition}\label{prop:cupRndense}
	For any $n\in\N$, $\bigcup_{m=n}^\infty \inter(\cR_{n,m})$ is dense in $\cM^T$.
\end{proposition}
\begin{proof}
	Fix $n\in\N$, $d\in\cM^T$ and $\nu > 0$. We will find $m\geq n$ and $\bar d\in \inter(\cR_{n,m})$ such that $\nm{\bar d - d}_\infty < \nu$. Set 
	\begin{align}\label{eq:choiceepsilon}
		\epsilon = \min(\frac{\nu}{5}, d(K,C_n)).
	\end{align}
	By \Cref{lm:nicecoverconstruct}, find an open cover $(U_i)_{i=0}^k$ of $K$ of order at most $\dim K$ (where $U_i\subseteq K$ is open in $K$), and a set $A = \{a_0,\ldots,a_k\} \subseteq K$, satisfying \ref{eq:aiinujiff} - \ref{eq:aiajfar} with respect to $d$ and $\epsilon$. By compactness, there exists $\xi > 0$ such that for each $x\in K$, there exists $i\in\{0,\ldots,k\}$, such that $B^d_\xi(x)\cap K \subseteq U_i$. For each $i$, define the closed set $U_i' = \{x\in U_i : d(x, K\setminus U_i) \geq \xi\}$. Evidently $K\subseteq \bigcup_{i=0}^k U_i'$ and, since $U_i'\subseteq U_i$ for each $i$, $a_i\in U_i'$ for each $i$ by \ref{eq:aiinujiff} as well. By \cite[Theorem 3.1.1]{engelking1995theory}, we can find sets $V_0',\ldots,V_k',$ which are open in $T$ such that $(V_i')_{i=0}^k$ covers $K$, has order at most $\dim K$, and satisfies $U_i' \subseteq V_i'$ for each $i$. Again for each $i$, find some $\eta_i > 0$ such that the open set 
	\begin{align}\label{eq:defVi}
		V_i = \{x\in V_i' : d(x, U_i') < \eta_i\} \cap B^d_{\frac{\epsilon}{2}}(a_i)
	\end{align}
	satisfies $a_j\not\in V_i$ whenever $j\not=i$, which is possible since $U_i'$ is closed and $a_j\not\in U_i \supseteq U_i'$ whenever $j\not=i$, by \ref{eq:aiinujiff}. Since $U_i' \subseteq U_i \subseteq B^d_{\frac{\epsilon}{2}}(a_i)$ by \ref{eq:duileqepsilon}, $U_i'\subseteq V_i$ for each $i$. Then $K\subseteq \bigcup_{i=0}^k U_i' \subseteq \bigcup_{i=0}^k V_i$. Also, as $V_i\subseteq V_i'$ for all $i$, $(V_i)_{i=0}^k$ has order at most $\dim K$, and furthermore satisfies \ref{eq:aiinujiff} and \ref{eq:duileqepsilon} (with $V_i$ in place of $U_i$).
	
	Next, by the compactness of $K$, find $\eta \in (0,\frac{\epsilon}{2})$ such that 
	\begin{align}\label{eq:defV}
		K\subseteq V:= \{x\in T : d(x,K) \leq \eta\} \subseteq \bigcup_{i=0}^k V_i.
	\end{align} 
	Now $(V_i\cap V)_{i=0}^k$ is an open cover of $V$ (inside $V$) of order at most $\dim K$, which, together with the set $A$, satisfies \ref{eq:aiinujiff} - \ref{eq:aiajfar} with respect to $d$ and $\epsilon$. From \Cref{prop:extmet} we obtain a metric $d_1\in\cM^V$ satisfying 
	\begin{align}\label{eq:d1ddist}
		\nm{d_1 - d\restrict _{V^2}}_\infty < 4\epsilon,
	\end{align} 
	and such that if $e\in\cM^V$ satisfies $\nm{e-d_1}_\infty \leq \frac{\epsilon}{12(\dim K+1)}$, then there exists an extension operator $E\colon\Lipo{A,e}\to\Lipo{V,e}$ with $\nm E \leq \Gamma$. By \cite[Theorem 1.1]{ishiki2020interpolation}, there exists an extension $d_2$ of $d_1$ to $T$ satisfying 
	\begin{align}\label{eq:d2ddistance}
		\nm{d_2 - d}_\infty \leq \nm{d_1 - d\restrict _{V^2}}_\infty < 4\epsilon.
	\end{align} 
	Therefore, for $e\in\cM^T$,
	\begin{align}\label{eq:eclosed2existexop}
		&\nm{e-d_2}_\infty \leq \frac{\epsilon}{12(\dim K+1)} \text{ implies the existence of an extension operator } \nonumber \\& E\colon\Lipo{A,e}\to\Lipo{V,e} \text{ satisfying } \nm E \leq \Gamma. 
	\end{align}
	Define the metric $e_1$ on $T$ by \begin{align}\label{eq:defe1pseudom}
		e_1(x,y) = \min(d(x,y), \eta),
	\end{align}
	and then define the metric $\bar d$ on $T$ by 
	\begin{align}\label{eq:defbard}
		\bar d = d_2 + \frac{\epsilon e_1}{14\eta(\dim K +1)}.
	\end{align}
	It is easy to see that $\bar d \in \cM^T$. 
	Since $\nm{e_1}_\infty \leq \eta$, we have \begin{align}\label{eq:distdbard2}
		\nm{\bar d - d_2}_\infty \leq \frac{\epsilon}{14(\dim K + 1)}.
	\end{align}
	Now define
	\begin{align}\label{eq:defV1V2}
		V_1 = \{x\in T : \bar d(x,K) \leq \frac{\epsilon}{32(\dim K + 1)}\} \;\text{ and }\; V_2 = \{x\in T : \bar d(x,K) <  \frac{\epsilon}{14(\dim K + 1)}\}.
	\end{align}
	If $x\in V_2$ then $e_1(x,K) < \eta$ by \eqref{eq:defbard}. Then $d(x,K) < \eta$ by \eqref{eq:defe1pseudom}, and so $x\in V$ by \eqref{eq:defV}. Therefore 
	\begin{align}\label{eq:V2inV}
		V_2 \subseteq V.
	\end{align} 
	Due to \eqref{eq:Cnchoice} and compactness, we can choose some $m \geq n$ such that 
	\begin{align}\label{eq:CncupV1isT}
		C_m \cup V_1 = T.
	\end{align}
	We have $\nm{\bar d - d}_\infty < 5\epsilon \leq \nu$ by \eqref{eq:distdbard2}, \eqref{eq:d2ddistance} and \eqref{eq:choiceepsilon}. We will show that $\bar d\in \inter(\cR_{n,m})$ by showing that $e\in \cR_{n,m}$ whenever $e\in\cM^T$ satisfies
	\begin{align}\label{eq:ebarddist}
		\nm{e-\bar d}_\infty < \frac{\epsilon}{480(\dim K + 1)}.
	\end{align}
	Choose $e\in\cM^T$ satisfying \eqref{eq:ebarddist} and set 
	\begin{align}\label{eq:defW12}
		W_1 = \{x\in T : e(x,K) \leq  \frac{\epsilon}{30(\dim K + 1)}\}  \text{ and } W_2 = \{x\in T : e(x,K) <  \frac{\epsilon}{15(\dim K + 1)}\}.
	\end{align} 
	By \eqref{eq:ebarddist}, \eqref{eq:defV1V2} and \eqref{eq:V2inV}, we have 
	\begin{align}\label{eq:longincl}
		V_1 \subseteq W_1 \subseteq W_2\subseteq V_2\subseteq V.
	\end{align}
	
	From \eqref{eq:distdbard2} and \eqref{eq:ebarddist} follows that
	\begin{align}\label{eq:ed2dist}
		\nm{e- d_2}_\infty \leq \nm{e-\bar d}_\infty + \nm{\bar d -d_2}_\infty < \frac{\epsilon}{12(\dim K + 1)}.
	\end{align}
	By \eqref{eq:eclosed2existexop}, there exists an extension operator $E\colon\Lipo{A,e} \to\Lipo{V,e}$ satisfying $\nm E \leq \Gamma$. Define the function $\rho\colon T\to [0,1]$ by 
	\begin{align}\label{eq:rhodef}
		\rho(x) = \min(1,\frac{30(\dim K + 1)}{\epsilon}e(x,W_1)).
	\end{align}
	As $\Lip_e(e(\cdot,W_1)) \leq 1$, we have \begin{align}\label{eq:liprhoest}
		\Lip_e(\rho) \leq \frac{30(\dim K + 1)}{\epsilon}.
	\end{align} 
	Note that $\rho(x) = 0$ for $x\in T\setminus C_m$, since $T\setminus C_m \subseteq V_1 \subseteq W_1$ by \eqref{eq:CncupV1isT} and \eqref{eq:longincl}. Furthermore, $\rho(x) = 1$ on $T\setminus V$. Indeed, if $x\in T\setminus V$ then $x\not\in W_2$ by \eqref{eq:longincl}. By \eqref{eq:defW12}, $e(x,K) \geq \frac{\epsilon}{15(\dim K + 1)},$ and then $$e(x,W_1)\geq \frac{\epsilon}{15(\dim K + 1)} - \frac{\epsilon}{30(\dim K + 1)} = \frac{\epsilon}{30(\dim K + 1)},$$ which implies $\rho(x) = 1$ by \eqref{eq:rhodef}. We conclude that \begin{align}\label{eq:supportsrhoincl}
		\supp(1-\rho) \subseteq V \;\text{ and }\; \supp(\rho)\subseteq C_m.
	\end{align}
	Furthermore, by \eqref{eq:defV}, \eqref{eq:defVi} and  \eqref{eq:choiceepsilon}, $$V\subseteq \bigcup_i V_i \subseteq \bigcup_i B_{\frac{\epsilon}{2}}^d(a_i) \subseteq T\setminus C_n,$$ which implies that $\rho(x) = 1$ for $x\in C_n$.

	We will now define an operator $\tildem H\colon \Lipo{A\cup C_m, e}\to \R^T$, where $\R^T$ is the set of all functions from $T$ to $\R$. For $f\in \Lipo{A\cup C_m,e}$, define $$\tildem H(f) = (1-\rho)E(f\restrict _A) + \rho f.$$ We have that $\tildem H$ is well defined by \eqref{eq:supportsrhoincl}, since the domain of $E(f\restrict _A)$ is $V$, and the domain of $f$ contains $C_m$. We also have that $\tildem H$ is linear, and satisfies $H(f)\restrict _{C_n\cup A} = f\restrict _{C_n\cup A}$ for every $f\in\Lipo{C_m\cup A,d}$, since $\rho(x) = 1$ for all $x\in C_n$, $\rho(x) = 0$ for all $x\in T\setminus C_m \supseteq K\supseteq A$, and $E$ is an extension operator. In particular $H(f)(a_0) = f(a_0) = 0$ for all $f\in\Lipo{C_m\cup A,d}$. Pick $f\in B_{\Lipo{C_m \cup A ,e}}$. We will estimate $\Lip_e(\tildem H(f))$. Set $g = E(f\restrict _A)$ and let $\tildem g$ be an extension of $g$ to $T$ by \cite[Theorem 1.33]{weaver}, so that $\Lip_e(\tildem g) \leq \nm E \leq \Gamma$. Similarly, let $\tildem f$ be an extension of $f$ to $T$ with $\Lip_e(\tildem f)\leq 1$. We consider the function $h\colon T\to\R$ given by $$h = (1-\rho)\tildem g + \rho \tildem f = (\tildem f - \tildem g)\rho + \tildem g.$$ It is clear that $\tildem H(f) = h$. Let us estimate $\Lip_e((\tildem f - \tildem g)\rho)$. Set $u = \tildem f - \tildem g$ and choose $x,y\in T$. If $x,y\in T\setminus V$ then $\rho(x) = \rho(y) = 1$ by \eqref{eq:supportsrhoincl}, and 
	\begin{align}\label{eq:urhowhenxynotinV}
		|u(x)\rho(x) - u(y)\rho(y)| = |u(x)-u(y)| \leq \Lip_e(u)e(x,y).
	\end{align}
	Otherwise, suppose without loss of generality that $x\in V$. Then by \eqref{eq:defV}, $x \in V_i$ for some $i\in\{0,\ldots,k\}$, and so $x\in B^d_{\frac{\epsilon}{2}}(a_i)$ by \eqref{eq:defVi}. Hence, 
	\begin{align*}
		e(x,a_i) &< \frac{\epsilon}{12(\dim K + 1)} + d_2(x,a_i) \tag*{by \eqref{eq:ed2dist}}\\&= \frac{\epsilon}{12(\dim K + 1)} + d_1(x,a_i) \tag*{since $d_2\restrict _{V^2} = d_1$,} \\&< \frac{9\epsilon}{2} + d(x,a_i) < 5\epsilon \tag*{by \eqref{eq:d1ddist}.}
	\end{align*}
	Note that $$u(a_i) = f(a_i) - g(a_i) = f(a_i) - E(f\restrict _A)(a_i) = 0,$$ since $E$ is an extension operator. Therefore $$|u(x)| = |u(x)-u(a_i)| \leq \Lip_e(u) e(x,a_i) \leq 5\epsilon\Lip_e(u).$$ Hence
	\begin{align*}
		|u(x)\rho(x) - u(y)\rho(y)| &\leq |u(x)||\rho(x)-\rho(y)| + |u(x)-u(y)|\rho(y) \\&\leq 5\epsilon\Lip_e(u)\Lip_e(\rho)e(x,y) + \Lip_e(u)e(x,y) \\&\leq (150\dim K + 151) \Lip_e(u)e(x,y) \tag*{ by \eqref{eq:liprhoest}.}
	\end{align*}
	By combining the last estimate with \eqref{eq:urhowhenxynotinV}, we conclude that $\Lip_e(u\rho) \leq (150\dim K + 151)\Lip_e(u).$ As $\Lip_e(\tildem g) \leq \Gamma$ and $\Lip_e(u) \leq \Gamma + 1$, we finally have $$\Lip_e(h) \leq (150\dim K + 151)\Lip_e(u) + \Gamma \leq (150\dim K + 152)(\Gamma + 1).$$ Hence $\tildem H(f) = h\in\Lipo{T,e}$. If we define $H\colon\Lipo{A\cup C_m,e}\to\Lipo{T,e}$ by $H(f) = \tildem H(f)$, then $\nm H \leq (150\dim K + 152)(\Gamma + 1)$. It is easy to deduce from the definition of $\tildem H$ and the fact that the domain of $E$ is a finite-dimensional space, that if $(f_j)_j$ is a sequence in $B_{\Lipo{C_m\cup A,e}}$ converging pointwise to $f\in B_{\Lipo{C_m\cup A,e}}$ then $(\tildem H(f_j))_j$ converges pointwise to $\tildem H(f)$. By \cite[Lemma 2.3]{stprbeforepub}, $H$ is a dual operator. We have proven that $e\in\cR_{n,m}$, and hence that $\bar d \in\inter(\cR_{n,m})$.
\end{proof}

We are now ready to give the proof of our main result.

\begin{proof}[Proof of \Cref{thm:main}]
	By \cite[Theorem 7.3.15]{engelking1995theory}, $k(T) = \emptyset$, and so $\trker T$ is an ordinal. The proof is by transfinite induction on $\trker T$. If $\trker T < \omega$ then $T = D_{\trker T} (T)$ and so $T$ is locally finite dimensional. By \cite[Proposition 5.5.11]{engelking1995theory}, $T$ is finite dimensional, and then $\cA^{T,1}$ is residual in $\cM^T$ by \Cref{prop:findimres}. Now fix $\alpha \geq \omega$ and suppose that for all compact metrisable spaces $S$ such that $k(S) =\emptyset$ and $\trker S < \alpha$, $\cA^{S,1}$ is residual in $\cM^S$. Let $T$ be a compact metrisable space satisfying $k(T) = \emptyset$ and $\trker T = \alpha$. If $\alpha$ is a limit ordinal then $$\emptyset = k(T) = T_\alpha = \bigcap_{\beta < \alpha} T_\beta,$$ where all $T_\beta$ are compact and nonempty, which is impossible. Therefore $l:= l(\alpha) > 0$ and $\beta:= \lambda(\alpha) < \alpha$. Note that $\beta \geq \omega$ since $\alpha \geq\omega$. We have $T_\beta\not=\emptyset$ and $T_\beta = D_l(T_\beta)$. Thus $T_\beta$ is locally finite dimensional and hence finite dimensional by \cite[Proposition 5.5.11]{engelking1995theory}. Set $K = T_\beta$, choose an arbitrary base point $a_0\in K$, and let $(C_n)_n$ be some increasing sequence of closed subsets of $T\setminus K$ satisfying \eqref{eq:Cnchoice}. By \cite[Lemma 7.3.7 (i)]{engelking1995theory}, $(C_n)_\beta \subseteq T_\beta = K$, so $(C_n)_\beta = \emptyset$ and therefore $\trker C_n \leq \beta < \alpha$. By the inductive hypothesis, $\cA^{C_n,1}$ is residual in $\cM^{C_n}$ for all $n\in\N$. We consider the set $$\cS = \cP^T \cap \{d\in\cM^T : d\restrict _{C_n^2} \in \cA^{C_n,1} \text{ for each } n\in\N\} \cap \bigcap_{n=1}^\infty \bigcup_{m=n}^\infty \inter(\cR_{n,m}),$$ where $\cR_{n,m}$ is defined by \eqref{eq:defRn} and \eqref{eq:defGamma}. It follows from \Cref{prop:p1uresidual,,prop:cupRndense} and \Cref{lm:residualitysubset} that $\cS$ is residual in $\cM^T$.

	Now suppose $d\in \cS$. We will prove that $\free{T,d}$ has the MAP. By the definition of $\cS$, there is a sequence $(k_n)_n \subseteq \N$, a sequence $(A_n)_n$ of finite subsets of $K$, and operators $H_n\colon\Lipo{C_{k_n}\cup A_n, d}\to\Lipo{T,d}$ such that, for each $n\in\N$, $a_0\in A_n$, $k_n\geq n$, $\free{C_n}$ has the MAP, and $H_n$ is a dual operator satisfying $\nm{H_n} \leq (150\dim K + 152)(\Gamma + 1)$ and $H(f)\restrict _{C_n\cup A_n} = f\restrict _{C_n\cup A_n}$ whenever $f\in\Lipo{C_{k_n}\cup A_n,d}$. Let $P_n\colon\free{T,d}\to\free{C_{k_n}\cup A_n,d}$ be the predual to $H_n$ for every $n$. By \cite[Theorem 3.7]{weaver}, the space $\free{C_{k_n}\cup A_n,d}$ can be viewed as a closed subspace of $\free{T,d}$. If $x\in C_n\cup A_n$ then $$(P_n(\delta_x))(f) = H_n(f)(x) = f(x) = \delta_x(f),$$ for all $f\in\Lipo{C_{k_n}\cup A_n,d}$. Then $P_n(\delta_x) = \delta_x$ for all $x\in C_n\cup A_n$, and so $P_n\restrict _{\free{C_n\cup A_n, d}}$ is the identity map on $\free{C_n\cup A_n,d}$. As $\beta \geq\omega$, $K =T_\beta \subseteq T_\omega$, and so $\inter(K) = \emptyset$ by \Cref{lm:kernelemptyinter}. This implies that $\bigcup_{n=1}^\infty C_n$ is dense in $T$. It is not hard to see that then $\bigcup_{n=1}^\infty \free{C_n\cup A_n,d}$ is dense in $\free{T,d}$. 
	
	To show that $\free{T,d}$ has the BAP, pick some $\mu_1,\ldots,\mu_k\in \free{T,d}$ and $\epsilon > 0$. Then find $n\in\N$ and $\nu_1,\ldots,\nu_k \in\free{C_n\cup A_n,d}$ such that $\nm{\mu_i - \nu_i} < \epsilon$ for every $i$. Since $A_n$ is finite, $\free{C_{k_n},d}$ is finite-codimensional in $\free{C_{k_n}\cup A_n,d}$. Since the former space has the MAP, $\free{C_{k_n}\cup A_n,d}$ has the BAP. Since $d\in \cP^T$, $(C_{k_n}\cup A_n,d)$ is purely 1-unrectifiable, and so by \Cref{thm:mapgrothendieck}, $\free{C_{k_n}\cup A_n,d}$ has the MAP. We can therefore find a finite-rank operator $Q$ on $\free{C_{k_n}\cup A_n,d}$ such that $\nm Q = 1$ and $\nm{Q(\nu_i) - \nu_i} < \epsilon$ for every $i$. We consider the finite-rank operator $\tildem Q = Q\circ P_n$ on $\free{T,d}$, which satisfies $\nm{\tildem Q} \leq (150\dim K + 152)(\Gamma + 1)$. Since $\nu_i\in\free{C_n\cup A_n,d}$, we have $P_n(\nu_i) = \nu_i$, and hence $\nm{\tildem Q(\nu_i) - \nu_i} < \epsilon$ for all $i$. But then $$\nm{\tildem Q(\mu_i) - \mu_i} \leq \nm{\tildem Q(\mu_i) - \tildem Q(\nu_i)} + \nm{\tildem Q(\nu_i) - \nu_i} + \nm{\nu_i - \mu_i} \leq ((150\dim K + 152)(\Gamma + 1)+2)\epsilon,$$ for all $i$. Therefore $\free{T,d}$ has the $(150\dim K + 152)(\Gamma + 1)$-BAP. 
	
	Since $d\in\cP^T$, by \cite[Theorem B]{purely1unrect}, $\free{T,d}$ is a dual space, and by \Cref{thm:mapgrothendieck}, $\free{T,d}$ has the MAP. Therefore $d\in\cA^{T,1}$, and so $\cS\subseteq\cA^{T,1}$. As $S$ is residual in $\cM^T$, $\cA^{T,1}$ is residual in $\cM^T$.
\end{proof}

\begin{remark}
	For a completely metrisable topological space $S$, by \cite[Corollary 7.3.14 and Theorem 7.3.15]{engelking1995theory}, $S$ is strongly countable-dimensional if and only if $k(S) = \emptyset$. This suggests that it might be difficult to generalise the argument by transfinite induction in the proof of \Cref{thm:main} to a more general setting than that of strongly countable-dimensional spaces.
\end{remark}

\begin{remark}
	Suppose that $S$ is a separable and metrisable topological space. If $k(S) = \emptyset$ then $\trker S < \omega_1$ by \cite[Corollary 7.3.6]{engelking1995theory}. Moreover, by \cite[Theorem 7.3.17 and Example 7.1.33]{engelking1995theory}, for each $\alpha < \omega_1$ there exists a compact, metrisable and strongly countable-dimensional space $S$ such that $\trker S \geq \alpha$. This shows that the transfinite induction argument in the proof of \Cref{thm:main} is needed.
\end{remark}

\subsection*{Acknowledgements}
	
I would like to thank Dr.~Richard Smith for his helpful comments and suggestions during the research process and the writing of the manuscript.
	
\bibliography{document}

\end{document}